\documentclass[12pt]{article}
\usepackage{amsmath,amsthm}
\usepackage{amsfonts,amsmath,comment}
\usepackage{amssymb}
\usepackage{fullpage}
\usepackage{epsf}
\usepackage{color}
\usepackage{graphicx}
\usepackage{pdflscape}
\usepackage{rotating}
\usepackage{rotating}

\newtheorem{theorem}{Theorem}[section]

\newtheorem{corollary}[theorem]{Corollary}

\theoremstyle{definition}

\newtheorem{example}[theorem]{Example}

\newtheorem{remark}[theorem]{Remark}

\renewcommand{\geq}{\geqslant}
\renewcommand{\leq}{\leqslant}

\def \md#1{{\,({\rm mod}\ #1)}}

\title{
Globally simple Heffter arrays $H(n;k)$ when $k\equiv 0,3\md{4}$
}

\author{
Kevin Burrage\thanks{Department of Computer Science, University of Oxford, UK; ARC Centre of Excellence for Mathematical and Statistical Frontiers, Queensland University of Technology (QUT), Australia (\texttt{kevin.burrage@qut.edu.au})}\and
Nicholas J. Cavenagh\thanks{Department of Mathematics, The University of Waikato, Private Bag 3105, Hamilton 3240, New Zealand.
\texttt{nickc@waikato.ac.nz}}\and
Diane M. Donovan\thanks{School of Mathematics and Physics, The University of Queensland,
 Queensland 4072,
Australia. (\texttt{dmd@maths.uq.edu.au})}\and
Emine \c{S}. Yaz\i c\i \thanks{Department of Mathematics, Ko\c{c} University, Sar{\i}yer,
34450, \.{I}stanbul, Turkey (\texttt{eyazici@ku.edu.tr})}
}

\begin{document}

\maketitle

\begin{abstract}
Square Heffter arrays are $n\times n$ arrays such that each row  and each column contains $k$ filled cells,
each row and column sum is divisible by $2nk+1$ and either $x$ or $-x$ appears in the array for each integer
$1\leq x\leq nk$.

Archdeacon noted that a Heffter array,  satisfying two additional conditions, yields a face $2$-colourable  embedding of the complete graph $K_{2nk+1}$ on an orientable surface, where for each colour, the faces give a $k$-cycle system.  Moreover, a cyclic permutation on the vertices acts as an automorphism of the embedding.
These necessary conditions pertain to cyclic orderings of the entries in  each  row and each column of the Heffter array and are: (1) for each row and each column the sequential partial sums determined by the cyclic ordering must be distinct modulo $2nk+1$; (2) the
composition of the cyclic orderings of the rows and columns is equivalent to a single cycle permutation on the entries in the array.
We construct Heffter arrays that satisfy condition (1) whenever (a) $k$ is divisible by $4$; or (b) $n\equiv 1\md{4}$ and $k\equiv 3\md{4}$; or (c)
$n\equiv 0\md{4}$, $k\equiv 3\md{4}$ and $n\gg k$.
As corollaries to the above we obtain pairs of orthogonal $k$-cycle decompositions of $K_{2nk+1}$.
\end{abstract}

\textbf{Keywords and MSC Code:} Heffter difference problem; Heffter  arrays;  orthogonal cycle decompositions. 05B30

\section{Introduction}
In 1896 Heffter, \cite{H}, introduced his now famous first difference problem:
 partition the set $\{1,\dots,3m\}$ into $m$ triples $\{a,b,c\}$ such that either $a+b=c$ or $a+b+c$ is divisible by $6m+1$. However, it was not until 1939 that Peltesohn, \cite{Pe}, showed that a solution exists whenever $m\neq 3$. A key interest in this problem is that solutions to Heffter's first difference problem yield cyclic Steiner triple systems; see \cite{CD}. A natural extension to this question is: can we identify a set of $m$ subsets $\{x_1,\dots,x_s\}\subset \{-ms,\dots, -1,1,\dots,ms\}$ such that the sum of the entries in each subset is divisible by $2ms+1$ and further if $x$ occurs in one of the subsets, $-x$ does not occur in any of the subsets? We call the set of $m$ such subsets a {\em Heffter system}. Two Heffter systems, $H_1=\{H_{11},\dots, H_{1m}\}$, $|H_{1i}|=s$ for $i=1,\dots m$, and  $H_2=\{H_{21},\dots, H_{2n}\}$ $|H_{2j}|=r$ for $j=1,\dots n$,  where $sm=nt$, are said to be {\em orthogonal} if for all $i,j$, $|H_{1i}\cap H_{2j}|\leq 1$. As observed by Dinitz and Mattern \cite{DM},  a Heffter system is equivalent to a
 {\em Heffter array} $H(m,n;s,t)$, which is an $m\times n$ array of integers, such that:
\begin{itemize}
\item each row contains $s$ filled cells and each column contains $t$ filled cells;
\item the elements in every row and column sum to $0$ in ${\mathbb Z}_{2ms+1}$; and
\item for each integer $1\leq x\leq ms$, either $x$ or $-x$ appears in the array.
\end{itemize}
Henceforth the set of integers $\{0,1,\dots ,n-1\}$ is denoted by $[n]$.
In the current paper the rows and columns of an $m\times n$ array will be indexed by $[m]$ and $[n]$, respectively.
A Heffter array is {\em square}, if $m = n$ and necessarily $s = t$, and is denoted $H(n;t)$.

The following is an example of a pair of orthogonal  Heffter systems that are equivalent to a Heffter array $H(6,12;8,4)$, given by Archdeacon in \cite{A}.

\begin{example}
Let $m=6$ and $s=8$. Then for each set
\begin{eqnarray*}
\begin{array}{ll}
\{-1 ,2  ,5  ,-6 , -25,26 ,29 ,-30\},&
\{3  ,-4 ,-7 ,8  , 27 ,-28,-31, 32\},\\
\{  9,-10,  -13,14 ,33 ,-34,   -37,38\},&
\{-11, 12,   15 ,-16,-35,36 ,  39 ,-40\},\\
\{-17, 18,21 ,-22,  -41,42 ,45 ,-46\}\ \mbox{\rm and} &
\{ 19,-20,-23,24 , 43 ,-44,-47,48 \},
\end{array}
\end{eqnarray*}
 its elements sum to zero. Also for each $x\in \{1,\dots ,48\}$,  precisely one of $x$ or $-x$ occurs in a subset. Thus these 6 subsets form a Heffter system.

Let $n=12$ and $t=4$. Then again for each of the following sets
\begin{eqnarray*}
\begin{array}{lll}
\{ 9,-11,-17,19\},&\{-10,12,18,-20 \},&\{ -1,3,21,-23  \},\\
\{2,-4,-22,24 \},&\{ 5,-7,-13,15\},&\{ -6,8,14,-16 \},\\
 \{33,-35,-41,43\},& \{-34,36,42,-44\},& \{-25,27,45,-47\},\\
  \{26,-28,-46,48\},&
  \{29,-31,-37,39\}\ \mbox{\rm and} & \{-30,32,38,-40\},
  \end{array}
  \end{eqnarray*}
 its elements sum to zero. Also for each $x\in \{1,\dots ,48\}$,  precisely one of $x$ or $-x$ occurs in a subset. Thus these 12 subsets form a Heffter system.

These two systems are orthogonal and thus we have equivalence with the following Heffter array $H(6,12;8,4)$.

\begin{scriptsize}
\begin{center}
\begin{tabular}{|c|c|c|c|c|c|c|c|c|c|c|c|}
\hline
    &   &-1 &2  &5  &-6 &   &   &-25&26 &29 &-30\\
\hline
    &   &3  &-4 &-7 &8  &   &   &27 &-28&-31& 32\\
\hline
  9 &-10&   &   &-13&14 &33 &-34&   &   &-37&38\\
\hline
-11 & 12&   &   &15 &-16&-35&36 &   &   &39 &-40\\
\hline
-17 & 18&21 &-22&   &   &-41&42 &45 &-46&   &\\
\hline
 19 &-20&-23&24 &   &   &43 &-44&-47&48 &   &\\
\hline

\end{tabular}
\end{center}
\end{scriptsize}
\end{example}

A cycle decomposition of a complete graph is the edge-disjoint decomposition of its edges into fixed length cycles. It was  Archdeacon \cite{A} who first showed that a Heffter array, together with a certain ordering of its elements, yields a biembedding of a pair of cycle decompositions of the complete graph onto an orientable surface. Since then a number of papers have appeared on the connection between  Heffter arrays and the biembedding of cycle decompositions as well as a number of papers studying more general biembeddings of the complete graph. See for examples
the papers \cite{A,BDF,CDP,CMPP,DM,GG,GrM,GM,M,V} and  \cite{ItalE}.

We next describe the above orderings.
Given a row $r$ of a Heffter array $H(m,n;s,t)$, if there exists a cyclic ordering $\phi_r=(a_0,a_1,\dots ,a_{s-1})$ of the entries of row $r$ such that, for $i=0,\dots, s-1$,
the partial sums
$$\alpha_i=\displaystyle\sum_{j=0}^{i} a_j \md{2ms+1}$$
 are all distinct, we say that $\phi_r$ is {\em simple}. A simple ordering of the entries of a column may be defined similarly.
If every row and column of a Heffter array $H(m,n;s,t)$ has a simple ordering, we say that the array is {\em simple}.
The existence of a simple $H(m,n;s,t)$ implies the existence of {\em orthogonal} decompositions ${\mathcal C}$ and ${\mathcal C}'$ of the graph $K_{2ms+1}$ into $s$-cycles and $t$-cycles (respectively);  that is, any cycle from ${\mathcal C}$ shares at  most one edge with any cycle from ${\mathcal C}'$ \cite{A}. Orthogonal cycle systems of the complete graph are studied in \cite{AHL}, \cite{CY} and \cite{CY2}. Observe that, if $s,t\leq5$ then any $H(m,n;s,t)$ is simple.

The composition of the cycles $\phi_r$,  for  $r\in [m]$,  is a permutation, denoted here  $\omega_r$, on the entries of the Heffter array. Similarly we may define the permutation
$\omega_c$ as the composition of the cycles $\phi_c$, for  $c\in [n]$.
If, the permutation $\omega_r\circ\omega_c$ can be written as a single cycle of length $ms=nt$, we say that
$\omega_r$ and $\omega_c$ are  {\em compatible} orderings for the Heffter array.

Archdeacon \cite{A}  proved the following theorem, showing that a Heffter array with a pair of compatible and simple orderings can be used to construct an embedding of the complete graph
$K_{2ms+1}$ on a surface.

\begin{theorem}\label{Archdeacon} {\rm \cite{A}} Suppose there exists a Heffter array $H(m, n; s, t)$ with orderings $\omega_r$ of the
symbols in the rows of the array and $\omega_c$ on the symbols in the columns of the array, where $\omega_r$ and $\omega_c$ are both simple and compatible.
Then there exists a face $2$-colourable embedding of $K_{2ms+1}$ on an orientable surface such that
the faces of one colour are cycles of length $s$ and the faces of the other colour are cycles of length $t$.
Moreover, in such an embedding the vertices may be labelled with the elements of $\mathbb{Z}_{2ms+1}$ such that the permutation $x \rightarrow  x + 1$ preserves the faces of
each colour.
\end{theorem}

If we relax the condition of simplicity in the above theorem, we still have a biembedding on an orientable surface but the faces collapse into smaller ones (and the cycles become circuits).
On the other hand if we relax only the condition of compatibility, we have an embedding onto a pseudosurface rather than surface, but ${\mathcal C}$ and ${\mathcal C}'$ remain orthogonal.

To date, the existence of Heffter arrays with orderings that are {\em both} compatible and simple is known in only a few specific cases:
 $H(3,n;n,3)$ \cite{DM};  $H(5,n;n,5)$ and $n\leq 100$ \cite{DM};
$H(n;t)$, $nt\equiv 3$ \md{4} and $t\in \{3,5,7,9\}$ \cite{ADDY, DW, CMPP}.

Ignoring orderings, in \cite{ABD} it was shown that
a $H(m,n;n,m)$ exists for all possible values of $m$ and $n$.
The spectrum for square Heffter arrays has been completely determined in \cite{ADDY}, \cite{DW} and \cite{CDDY}.

\begin{theorem}
There exists an $H(n; k)$ if and only if $3 \leq k \leq  n$.
\label{mainthm}
\end{theorem}

For the sake of ease of description, Heffter arrays often possess some extra properties that we now describe.
A Heffter array is called an {\em integer} Heffter array if the sum of each
row and column is $0$ in ${\mathbb Z}$.
Suppose that a simple cyclic ordering  $\phi_r=(a_1,a_2,\dots ,a_s)$ of a row $r$ of a Heffter array has the property that whenever entry $a_i$ lies in cell $(r,c)$ and entry $a_{i+1}$ lies in cell
$(r,c')$, then $c<c'$. That is, the ordering for the row $r$ is taken from left to right across the array. We say that $\phi_r$ is the {\em natural} ordering for the rows and define a natural column ordering in a similar way from top to bottom.
If the natural ordering for every row and column is also a simple ordering, we say that the Heffter array is {\em globally simple}.

We focus on square Heffter arrays in this paper and now can state our main results.
\begin{theorem}
If $p>0$ and $n\geq 4p$ then there exists a globally simple integer Heffter array $H(n;4p)$.
\label{main1}
\end{theorem}

We prove Theorem \ref{main1} in Section 2.

\begin{corollary}
If $p>0$ and $n\geq 4p$, there exists a pair of orthogonal decompositions of $K_{8np+1}$ into cycles of length $4p$.
\end{corollary}

\begin{theorem}
Let $n\equiv 1 \md{4}$, $p>0$ and $n\geq 4p+3$, then there exists a globally simple integer Heffter array $H(n;4p+3)$.
\label{main2}
\end{theorem}

\begin{theorem}
Let $n\equiv 0\md{4}$. Then there exist constants $c$ and $n_0$ such that if $n\geq n_0$ and $n-4p\geq c\log^2{n}$,
 then there is a globally simple Heffter array $H(n;4p+3)$.
\label{main3}
\end{theorem}

We prove Theorems \ref{main2} and \ref{main3} in Sections 3 and 4.

\begin{corollary}
If either {\rm (a)} $n\equiv 1 \md{4}$, $p>0$ and $n\geq 4p+3$, or {\rm (b)} $n\equiv 0\md{4}$ and $n\gg 4p$,
 there exists a pair of orthogonal decompositions of $K_{2n(4p+3)+1}$ into cycles of length $4p+3$.
\end{corollary}

 Even though not explicitly stated, the partial sums (given by the natural ordering) will also be distinct  $\md{2nk+2}$.
As shown in \cite{CMPP}, our results thus also yield orthogonal cycle decompositions of the complete graph of order $2nk+2$ minus a $1$-factor.

The following are useful conventions and results which will be used through out the paper. It is important to be aware that row and column indices are {\em always} calculated modulo $n$, while entries of arrays are {\em always} evaluated as integers.
The {\em support} of an array $A$ is the set containing the absolute values of the entries  of  $A$ and denoted $s(A)$.
In what follows, for a partially filled array $A=[A(i,j)]$  we use $A(i,j)$ to denote the  entry  in cell $(i,j)$ of array $A$.
The cells of an $n\times n$ array can be partitioned into $n$ disjoint {\em diagonals} $D_d$, $d\in [n]$, where \begin{eqnarray*}
D_d:=\{(i+d,i)\mid i\in [n]\}.
\end{eqnarray*}
Let the entry in  row $a$ and column $a$ of diagonal $D_i$ be denoted by $d_i(r_a)$ and $d_i(c_a)$, respectively, with these values defined to be $0$ when there is no entry.
For a given row $a$ we define  $\Sigma(x)=\sum_{i=0}^x d_i(r_a)$ and
for a given column $a$ we define  $\overline{\Sigma}(x)=\sum_{i=0}^x d_i(r_a)$.
For a given row $a$, the values of $\Sigma(x)$ such that $d_x(r_a)$ is non-zero are called the {\em row partial sums} for $a$.
For a given column $a$, the values of $\overline{\Sigma}(x)$ such that $d_x(c_a)$ is non-zero are called the {\em column partial sums} for $a$.
Thus to show an array is globally simple, it suffices to show that the row partial sums are distinct (modulo $2nk+1$) for each row $a$ and that the column partial sums are distinct (modulo $2nk+1$) for each column $c$.
To aid the reader, we will often refer to the following straightforward observations.
\begin{remark} Let $m,x_1,x_2,\alpha_1,\alpha_2,\beta_1,\beta_2$ be integers and $m>0$. Then for:
 \begin{eqnarray}
 -m\leq x_1,x_2\leq m,&& x_1\equiv x_2\md{2m+1}\Rightarrow x_1=x_2;  \label{mods}\\
 0\leq x_1,x_2< m,&& x_1\equiv x_2\md{m} \Rightarrow x_1=x_2; \label{modl}\\
 -\frac{m}{2}<\alpha_1,\alpha_2<\frac{m}{2},&& \beta_1m+\alpha_1=\beta_2m+\alpha_2\Rightarrow \beta_1=\beta_2\ and\ \alpha_1=\alpha_2; \label{n}\\
 -m<x_1<0<x_2<m,&& x_1\equiv x_2\md{m} \Rightarrow x_2=m+x_1.  \label{mod=}
\end{eqnarray}
\end{remark}

\section{Globally simple integer $H(n;4p)$ constructions}
In this section we prove Theorem \ref{main1}. That is, we construct a globally simple Heffter array $H(n;4p)$ for each $n$ and $p\geq3$ such that $n\geq 4p$. Note that a globally simple $H(n,8)$ was constructed in \cite{CMPP} and it is easy to see that all Heffter arrays $H(n,4)$ are globally simple. We will divide this section according to the parity of $p$. Throughout this section $k=4p$.
\subsection{$p$ is odd}

Let $p>1$ be odd and $I=[\frac{p-1}{2}]$. We remind the reader that throughout this paper, rows and columns are evaluated modulo $n$, while entries are always evaluated as integers.

For $x\in [n]$ and $i\in I$ define the array $A$ to have the following entries:
\begin{eqnarray*}
4i+1+kx&\mbox{ in cell}&(4i+x,x)\in  D_{4i}\nonumber,\\
-(4i+2)-k(x+1\md{n}) &\mbox{ in cell}&(4i+1+x,x)\in D_{4i+1},\\
-(k-(4i+3))-kx    &\mbox{ in cell}&(4i+2+x,x)\in  D_{4i+2},\\
k-(4i+4)+k(x+1\md{n})&\mbox{ in cell}&(4i+3+x,x)\in  D_{4i+3},
\end{eqnarray*}\vspace{-0.7cm}\begin{eqnarray*}
2p-1+kx &\mbox{ in cell}&(2p-2+x,x)\in  D_{2p-2},\\
-2p-k(x+1\md{n}) &\mbox{ in cell}&(2p-1+x,x)\in  D_{2p-1},
\end{eqnarray*}\vspace{-0.7cm}\begin{eqnarray*}
-(2p-2-4i)-kx &\mbox{ in cell}&(2p+4i+x,x)\in  D_{2p+4i},\\
2p-3-4i+k(x+1 \md{n})&\mbox{ in cell}&(2p+4i+1+x,x)\in  D_{2p+4i+1},\\
2p+4+4i+kx &\mbox{ in cell}&(2p+4i+2+x,x)\in  D_{2p+4i+2}, \\
-(2p+5+4i)-k(x+1\md{n}) &\mbox{ in cell}&(2p+4i+3+x,x)\in  D_{2p+4i+3},
\end{eqnarray*}\vspace{-0.7cm}\begin{eqnarray*}
-(2p+1+kx) &\mbox{ in cell}&(k-2+x,x)\in  D_{k-2},\\
k+k(x+1\md{n}) &\mbox{ in cell}&(k-1+x,x)\in  D_{k-1}.\\
\end{eqnarray*}
\begin{example} A globally simple Heffter array $H(17;12)$ ($n=17$ and $p=3$).
\begin{scriptsize}
\begin{center}
\begin{tabular}{|c|c|c|c|c|c|c|c|c|c|c|c|c|c|c|c|c|}
\hline
1& & & & & & 96&-91&-119&118&135&-136&-162&161&188&-189&-2\\ \hline
-14&13& & & & & & 108&-103&-131&130&147&-148&-174&173&200&-201\\ \hline
-9&-26&25& & & & & & 120&-115&-143&142&159&-160&-186&185&8\\ \hline
20&-21&-38&37& & & & & & 132&-127&-155&154&171&-172&-198&197\\ \hline
5&32&-33&-50&49& & & & & & 144&-139&-167&166&183&-184&-6\\ \hline
-18&17&44&-45&-62&61& & & & & & 156&-151&-179&178&195&-196\\ \hline
-4&-30&29&56&-57&-74&73& & & & & & 168&-163&-191&190&3\\ \hline
15&-16&-42&41&68&-69&-86&85& & & & & & 180&-175&-203&202\\ \hline
10&27&-28&-54&53&80&-81&-98&97& & & & & & 192&-187&-11\\ \hline
-23&22&39&-40&-66&65&92&-93&-110&109& & & & & & 204&-199\\ \hline
-7&-35&34&51&-52&-78&77&104&-105&-122&121& & & & & &12 \\ \hline
24&-19&-47&46&63&-64&-90&89&116&-117&-134&133& & & & &  \\ \hline
 &36&-31&-59&58&75&-76&-102&101&128&-129&-146&145& & & &  \\ \hline
 & &48&-43&-71&70&87&-88&-114&113&140&-141&-158&157& & &  \\ \hline
 & & &60&-55&-83&82&99&-100&-126&125&152&-153&-170&169& &  \\ \hline
 & & & &72&-67&-95&94&111&-112&-138&137&164&-165&-182&181&  \\ \hline
 & & & & &84&-79&-107&106&123&-124&-150&149&176&-177&-194&193 \\ \hline
\end{tabular}
\end{center}
\end{scriptsize}
\end{example}

\subsubsection{Support of the array when $p$ is odd}
Observe that for each $\alpha\in[k]$, $\displaystyle s(D_{\alpha})=\{ S_\alpha+kx| x\in [n]\}$ in $A$ where $S_\alpha$ satisfies:

$
\begin{array}{ll}
\{S_{4i} | i\in I\}=\{1,5,\dots ,2p-5\},&\{S_{4i+2} | i\in I\}=\{2p+3,2p+7,\dots,4p-3\},\\
\{S_{4i+1}| i\in I\}=\{2,6,\dots,2p-4\},&\{S_{4i+3} | i\in I\}=\{2p+2,2p+6,\dots,4p-4\},\\
\{S_{2p+4i} | i\in I\}=\{4,8,\dots,2p-2\},&\{S_{2p+4i+2} | i\in I\}=\{2p+4,2p+8,\dots,4p-2\}, \\
\{S_{2p+4i+1}| i\in I\}=\{3,7,\dots,2p-3\},&\{S_{2p+4i+3} | i\in I\}=\{2p+5,2p+9,\dots,4p-1\}, \\
\end{array}$

$\begin{array}{llll}
S_{2p-2}=2p-1,&S_{4p-2}=2p+1,&
S_{2p-1}=2p,& S_{4p-1}=4p.\\
\end{array}
$

Hence it is easy to see that $s(A)=\{1,2,\dots,nk\}$.

\subsubsection{Distinct column partial sums when $p$ is odd}

Recall that $k=4p$. In this section we will show that in the array $A$, $|\overline{\Sigma}(\alpha)|\leq nk$ for all $\alpha\in [k]$. Hence by (\ref{mods}), $$\overline{\Sigma}(\alpha_1)\equiv \overline{\Sigma}(\alpha_2)\md{2nk+1} \Rightarrow \overline{\Sigma}(\alpha_1)=\overline{\Sigma}(\alpha_2).$$
But then we will show that $\overline{\Sigma}(\alpha_1)\not=\overline{\Sigma}(\alpha_2)$ by comparing these values modulo $k$. Hence we obtain the required result $\overline{\Sigma}(\alpha_1)\not\equiv \overline{\Sigma}(\alpha_2)\md{2nk+1}$.

First observe that for each column $a$ and $i\in I$:

$\displaystyle \sum_{j=0}^3 d_{4i+j}=-2$,  and  $\displaystyle \sum_{j=0}^3 d_{2p+4i+j}=-2.$

Now the partial column sums for each column $a$ can be calculated as follows:
\begin{eqnarray*}
\overline{\Sigma}(4i)&=&4i+1+ak-2i=2i+1+ak<nk\\
&\Rightarrow&\{\overline{\Sigma}(4i)\md{k} \mid i\in I\}=\{1,3,\dots, p-2\}.\\
\overline{\Sigma}(4i+1)&=&2i+1+ak-(4i+2)-k(a+1\md{n})\\
&=&-(2i+1)+ka-k(a+1\md{n}),\\
&\Rightarrow&\{\overline{\Sigma}(4i+1)\md{k} \mid i\in I\}=\{3p+2,\dots, 4p-3, 4p-1\}.\\
\overline{\Sigma}(4i+2)&=&-(2i+1)+ka-k(a+1\md{n})-k+(3+4i)-ak\\
&=&2i+2-k-k(a+1\md{n}),\\
&\Rightarrow&\{\overline{\Sigma}(4i+2)\md{k} \mid i\in I\}=\{2,4,\dots,p-1\}.\\
\overline{\Sigma}(4i+3)&=&2i+2-k-k(a+1\md{n})+k-(4+4i)+k(a+1\md{n})\\
&=&-2i-2,\\
&\Rightarrow&\{\overline{\Sigma}(4i+3)\md{k} \mid i\in I\}=\{3p+1,\dots, 4p-4,4p-2\}.\\
\overline{\Sigma}(2p-2)&=&-2(p-3)/2-2+2p-1+ak=p+ak.\\
\overline{\Sigma}(2p-1)&=&p+ak-2p-k(a+1\md{n})=-p+ak-k(a+1\md{n}).\\
\overline{\Sigma}(2p+4i)&=&-p+ak-k(a+1\md{n})-2p+2+4i-ak-2i\\
&=&-3p+2i+2-k(a+1\md{n}),\\
&\Rightarrow&\{\overline{\Sigma}(2p+4i)\md{k} \mid i\in I\}=\{p+2,\dots 2p-3,2p-1\}.\\
\overline{\Sigma}(2p+4i+1)&=&-3p+2i+2-k(a+1\md{n})+2p-3-4i+k(a+1\md{n})\\
&=&-p-(2i+1),\\
&\Rightarrow&\{\overline{\Sigma}(2p+4i+1)\md{k}\mid i\in I\}=\{2p+2,\dots ,3p-3, 3p-1\}.\\
\overline{\Sigma}(2p+4i+2)&=&-p-(2i+1)+2p+4+4i+ak=p+2i+3+ak,\\
&\Rightarrow&\{\overline{\Sigma}(2p+4i+2)\md{k}\mid i\in I\}=\{p+3,p+5,\dots,2p\}.\\
\overline{\Sigma}(2p+4i+3)&=&p+2i+3+ak-2p-5-4i-k(a+1\md{n}),\\
&=&-p-2(i+1)+ak-k(a+1\md{n})\\
&\Rightarrow&\{\overline{\Sigma}(2p+4i+4)\ mod\ k\mid i\in I\}=\{2p+1,\dots ,3p-4,3p-2\}.\\
\overline{\Sigma}(k-2)&=&-p-2((p-3)/2+1)+ak-k(a+1\md{n})-2p-1-ak\\
&=&-k-k(a+1\md{n})\neq 0. \\
\overline{\Sigma}(k-1)&=&0.
\end{eqnarray*}


One can easily check from above calculations that for column $a$, $\overline{\Sigma}(\alpha_1)\neq \overline{\Sigma}(\alpha_2) \md{k}$ for all $\alpha_1,\alpha_2\in[k-1]$ and $\alpha_1\neq \alpha_2$. Furthermore $\overline{\Sigma}(k-2)=-d_{k-1}(c_a)\neq 0$. Also it is not hard to check that for all $\alpha\in[k]$,  $|\overline{\Sigma}(\alpha)|\leq nk$. Hence all the column partial sums are distinct $\md{2kn+1}$.

\subsubsection{Distinct row partial sums when $p$ is odd}
As elements in $s(D_\alpha)$ are all congruent modulo $k$ in $A$, 
 we have $d_{\alpha}(r_a)\equiv d_\alpha(c_b) \md{k}$  for all $a,b\in[n], \alpha\in [k]$. Hence $\displaystyle{ \Sigma_{j=0}^\alpha d_j(r_a)\equiv \Sigma_{j=0}^\alpha d_j(c_b)\md{k}}$. Now as the partial column sums up to and including diagonal $4p-2$ are distinct modulo $k$, partial sums of rows up to and including diagonal $4p-2$ are distinct modulo $k$.
To use the same argument as above, we thus just need to show that $|\Sigma(\alpha)|\leq nk$ for each row $a$ and $\alpha\in [k]$.

First observe that for each $0\leq j\leq 2p-2$, $d_{2j}(r_a)$ and $d_{2j+1}(r_a)$ are in the form:

$d_{2j}(r_a)=\alpha +s_jk(a-\beta\md{n})$ and $d_{2j+1}(r_a)=-\alpha-1 -s_jk(a-\beta\md{n})$ where $\alpha$ and $\beta$ are integers and $s_j\in\{1,-1\}$.

Hence $d_{2j}(r_a)+d_{2j+1}(r_a)=-1$ and $\Sigma(2j+1)=-(j+1)$,  for each $0\leq j\leq 2p-2$.

Now $d_{4i}(r_a),d_{2p-2}(r_a), d_{2p+4i+2}\geq 0$. Hence $-nk\leq \Sigma(2\alpha)=-\alpha+d_{2\alpha}(r_a)\leq nk$ for $\alpha\in\{2i,p-1,p+2i+1|i\in I\}. $

Finally
\begin{eqnarray*}
\Sigma(4i+2)&=&-(2i+1)-k+4i+3-k(a-4i-2\md{n})\\
&=&2i+2-k-k(a-4i-2\md{n})\geq -nk,\\
\Sigma(2p+4i)&=&-(p+2i)-2p+2+4i-k(a-4i-2p\md{n})\\
&=&-3p+2i+2-k(a-4i-2p\md{n})\geq -nk,\\
\Sigma(4p-2)&=&-(2p-1)-2p-1-k(a-k+2\md{n})\\
&=&-k-k(a-k+2\md{n})\geq -nk, \\
\Sigma(4p-1)&=&0.
\end{eqnarray*}

\subsection{$p$ is even}

Let $p>2$ be even and $I=[\frac{p-2}{2}]$.

For $x\in [n]$ and $i\in I$ define the array $A$ to have the following entries:
\begin{eqnarray*}
4i+1+kx&\mbox{ in cell}&(4i+x,x)\in  D_{4i}\nonumber,\\
-(4i+2)-k(x+1\md{n}) &\mbox{ in cell}&(4i+1+x,x)\in D_{4i+1},\\
-(k-(4i+3))-kx    &\mbox{ in cell}&(4i+2+x,x)\in  D_{4i+2},\\
k-(4i+4)+k(x+1\md{n})&\mbox{ in cell}&(4i+3+x,x)\in  D_{4i+3},
\end{eqnarray*}\vspace{-0.7cm}\begin{eqnarray*}
2p-3+kx &\mbox{ in cell}&(2p-4+x,x)\in  D_{2p-4},\\
-2p+2-k(x+1\md{n}) &\mbox{ in cell}&(2p-3+x,x)\in  D_{2p-3},
\end{eqnarray*}\vspace{-0.7cm}\begin{eqnarray*}
-(2p-4i)-kx &\mbox{ in cell}&(2p+4i-2+x,x)\in  D_{2p+4i-2},\\
2p-1-4i+k(x+1 \md{n})&\mbox{ in cell}&(2p+4i-1+x,x)\in  D_{2p+4i-1},\\
2p+2+4i+kx &\mbox{ in cell}&(2p+4i+x,x)\in  D_{2p+4i}, \\
-(2p+3+4i)-k(x+1\md{n}) &\mbox{ in cell}&(2p+4i+1+x,x)\in  D_{2p+4i+1},
\end{eqnarray*}\vspace{-0.7cm}\begin{eqnarray*}
-4-kx &\mbox{ in cell}&(k-6+x,x)\in  D_{k-6},\\
3+k(x+1 \md{n})&\mbox{ in cell}&(k-5+x,x)\in  D_{k-5},\\
k-2+kx &\mbox{ in cell}&(k-4+x,x)\in  D_{k-4}, \\
-(k-1)-k(x+1\md{n}) &\mbox{ in cell}&(k-3+x,x)\in  D_{k-3},\\
\\
-(2p+1+kx) &\mbox{ in cell}&(k-2+x,x)\in  D_{k-2},\\
k+k(x+1\md{n}) &\mbox{ in cell}&(k-1+x,x)\in  D_{k-1},\\
\end{eqnarray*}
\begin{example} A globally simple Heffter array $H(17,16)$ ($n=17$ and $p=4$).
\begin{scriptsize}
\begin{center}
\begin{tabular}{|c|c|c|c|c|c|c|c|c|c|c|c|c|c|c|c|c|}
\hline
1& &64 &-57 &-95 &94 & 115&-91&-119&118&183&-184&-214&213&252&-253&-2\\ \hline
-18&17& &80 &-73 &-111&110& 131&-103&-131&130&199&-200&-230&229&268&-269\\ \hline
-13&-34&33& &96 &-89 &-127 &126 & 147&-115&-143&142&215&-216&-246&245&12\\ \hline
28&-29&-50&49& &112 &-105 &-143 &142&163&-127&-155&154&231&-232&-262&261\\ \hline
5&44&-45&-66&65& &128 &-121 &-159 &158& 179&-139&-167&166&247&-248&-6\\ \hline
-22&21&60&-61&-82&81& &144 &-137 &-175 &174 & 195&-151&-179&178&263&-264\\ \hline
-8&-38&37&76&-77&-98&97& &160 & -153&-191&190 & 211&-163&-191&190&7\\ \hline
23&-24&-54&53&92&-93&-114&113& &176 &-169& -207&206 & 227&-175&-203&202\\ \hline
10&39&-40&-70&69&108&-109&-130&129& &192 &-185&-223 &222 & 243&-187&-11\\ \hline
-27&26&55&-56&-86&85&124&-125&-146&145& &208 &-201&-239 &238 & 259&-199\\ \hline
-4&-43&42&71&-72&-102&101&140&-141&-162&161& &224 & -217&-255 &254 &3 \\ \hline
19&-20&-59&58&87&-88&-118&117&156&-157&-178&177& &240 &-233&-271 &270  \\ \hline
14 &35&-36&-75&74&103&-104&-134&133&172&-173&-194&193& &256 &-249 &-15  \\ \hline
-31 &30 &51&-52&-91&90&119&-120&-150&149&188&-189&-210&209& &272 &-265  \\ \hline
-9 &-47 &46 &67&-68&-107&106&135&-136&-166&165&204&-205&-226&225& &16  \\ \hline
32 &-25&-63 &62 &83&-84&-123&122&151&-152&-182&181&220&-221&-242&241&  \\ \hline
   &48 &-41&-79 &78&99&-100&-139&138&167&-168&-198&197&236&-237&-258&257 \\ \hline
\end{tabular}
\end{center}
\end{scriptsize}
\end{example}

\subsubsection{Support when $p$ is even}

Observe that for each $\alpha\in[k]$, $\displaystyle s(D_{\alpha})= \{S_\alpha +kx | x\in [n]\}$ in $A$ where $S_\alpha$ satisfies:

$\begin{array}{ll}
\{S_{4i}|i\in I\}=\{1,5,\dots ,2p-7\},&\{S_{4i+2}|i\in I\}=\{2p+5,2p+9,\dots, 4p-3\}, \\
\{S_{4i+1}|i\in I\}=\{2,6,\dots ,2p-6\},& \{S_{4i+3}|i\in I\}=\{2p+4,2p+8,\dots, 4p-4\},\\
\{S_{2p+4i-2}|i\in I\}=\{8,12,\dots, 2p\},&\{S_{2p+4i}|i\in I\}=\{2p+2,2p+6,\dots, 4p-6\},\\
\{S_{2p+4i-1}|i\in I\}=\{7,11,\dots, 2p-1\},&\{S_{2p+4i+1}|i\in I\}=\{2p+3,2p+7,\dots, 4p-5\},\\
\end{array}$

$
\begin{array}{llll}
S_{2p-4}=2p-3,&S_{4p-2}=2p+1,&
S_{2p-3}=2p-2,&S_{4p-1}=4p,\\
S_{4p-6}=4,&S_{4p-4}=4p-2,&
S_{4p-5}=3,& S_{4p-3}=4p-1.\\
\end{array}$



Hence it is easy to see that $s(A)=\{1,2,\dots,nk\}$.

\subsubsection{Distinct column partial sums when $p$ is even}

First observe that for each column $a$ and $i\in I$:

$\displaystyle \sum_{j=0}^3 d_{4i+j}(c_a)=-2$,  and  $\displaystyle \sum_{j=0}^3 d_{2p+4i-2+j}(c_a)=-2.$

Similarly to the previous subsection, the column partial sums for each column $a$ can be calculated as follows:
\begin{eqnarray*}
\overline{\Sigma}(4i)&=&2i+1+ka \Rightarrow\{\overline{\Sigma}(4i)\md{k}\mid i\in I\}=\{1,3,\dots, p-3\},\\
\overline{\Sigma}(4i+1)&=&-(2i+1)+ka-k(a+1\md{n})\\
&\Rightarrow&\{\overline{\Sigma}(4i+1)\md{k}\mid i\in I\}=\{3p+3,\dots, 4p-3, 4p-1\},\\
\overline{\Sigma}(4i+2)&=&2i+2-k-k(a+1\md{n})\\
&\Rightarrow&\{\overline{\Sigma}(4i+2)\md{k}\mid i\in I\}=\{2,4,\dots,p-2\},\\
\overline{\Sigma}(4i+3)&=&-2i-2\Rightarrow\{\overline{\Sigma}(4i+3)\md{k}\mid i\in I\}=\{3p+2,\dots, 4p-4,4p-2\},\\
\overline{\Sigma}(2p-4)&=&-2(p-4)/2-2+2p-3+ak=p-1+ka,\\
\overline{\Sigma}(2p-3)&=&-p+1+ka-k(a+1\md{n}),\\
\overline{\Sigma}(2p+4i-2)&=&-3p+2i+1-k(a+1\md{n})\\
&\Rightarrow&\{\overline{\Sigma}(2p+4i-2)\md{k} \mid i\in I\}=\{p+1,\dots ,2p-1,2p-3\},\\
\overline{\Sigma}(2p+4i-1)&=&-p-2i,\\
&\Rightarrow&\{\overline{\Sigma}(2p+4i-1)\md{k} \mid i\in I\}=\{2p+4,\dots ,3p-2, 3p\},\\
\overline{\Sigma}(2p+4i)&=&-p-2i+2p+4i+ak+2=p+2i+2+ka\\
&\Rightarrow&\{\overline{\Sigma}(2p+4i+1)\md{k} \mid i\in I\}=\{p+2,p+4,\dots, 2p-2\},\\
\overline{\Sigma}(2p+4i+1)&=&-p-2i-1+ka-k(a+1\md{n})\\
&\Rightarrow&\{\overline{\Sigma}(2p+4i+4)\md{k} \mid i\in I\}=\{2p+3,\dots ,3p-3,3p-1\},\\
\overline{\Sigma}(k-6)&=&-(2p+1)-k(a+1\md{n}),\\
\Sigma(k-5)&=&-2p+2,\\
\Sigma(k-4)&=& 2p+ka,\\
\Sigma(k-3)&=&-2p+1+ka-k(a+1\md{n}),\\
\Sigma(k-2)&=&-k-k(a+1\md{n})\neq 0,\\
\Sigma(k-1)&=&0.
\end{eqnarray*}

One can easily check from above calculations that for column $a$, $\overline{\Sigma}(\alpha_1)\neq \overline{\Sigma}(\alpha_2)\md{k}$ for all $\alpha_1,\alpha_2\in[k-1]$ and $\alpha_1\neq \alpha_2$. It is also straightforward to check that $|\overline{\Sigma}(\alpha)|\leq nk$. Hence all the column partial sums are distinct modulo $2kn+1$.

\subsubsection{Distinct row partial sums when $p$ is even}

Similarly to the case when $p$ is odd, we just need to show that $|\Sigma(\alpha)|\leq nk$ for each row $a$ and $\alpha\in [k]$. As before, $d_{2j}(r_a)+d_{2j+1}(r_a)=-1$ and $\Sigma(2j+1)=-(j+1)$,  for each $0\leq j\leq 2p-2$.

Now $d_{4i}(r_a),d_{2p-4}(r_a), d_{2p+4i}(r_a), d_{k-4}(r_a)\geq 0$ for $i\in I$. Hence $-nk\leq \Sigma(2\alpha)=-\alpha+d_{2\alpha}(r_a)\leq nk$ for $\alpha\in\{2i,p-2,p+2i,2p-2|i\in I\}.$

Finally,
\begin{eqnarray*}
\Sigma(4i+2)&=&-(2i+1)-k+4i+3-k(a-4i-2\md{n}),\\
&=&2i+2-k-k(a-4i-2\md{n})\geq -nk,\\
\Sigma(2p+4i-2)&=&-(p+2i-1)-2p+4i-k(a-4i-2p+2\md{n}),\\
&=&-3p+2i+1-k(a-4i-2p\md{n})\geq -nk,\\
\Sigma(k-6)&=&-(2p-3)-4-k(a-4i-2p\md{n})\geq -nk,\\
\Sigma(k-2)&=&-(2p-1)-2p-1-k(a-k+2\md{n}) \\
&=&-k-k(a-k+2\md{n})\geq -nk,\\
\Sigma(k-1)&=&0.
\end{eqnarray*}

So Theorem 1.4 is proven.

\section{Support shifted globally simple Heffter arrays}\label{thmex}

The array $A$ is defined to be a {\em support shifted Heffter array} $H(n;4p,\gamma)$ if it satisfies the following properties:
\begin{itemize}
\item[{\bf P1.}] Every row and every column of $A$ has $4p$ filled cells.
\item[{\bf P2.}] $s(A)=\{\gamma n+1,\dots,(4p+\gamma)n\}$.
\item[{\bf P3.}] Elements in every row and every column sum to $0$.
\item[{\bf P4.}] Partial sums are distinct in each row and each column of $A$  modulo $2(4p+\gamma)n+1$.
\end{itemize}

A related generalization of Heffter arrays is studied in \cite{ItalE2}.
Note that a support shifted Heffter array $H(n;4p,0)$ is in fact an integer Heffter array $H(n;4p)$.
In the following section we let $\gamma=3$ and we merge
the support shifted Heffter array $H(n;4p,3)$ constructed below with a Heffter array $H(n;3)$ to obtain Heffter arrays $H(n;4p+3)$.
In this section we write our results generally in terms of $\gamma$ in case the following theorem is useful for future research.

\begin{theorem}\label{main}
Let $p>0$, $n\geq 4p$, and $\gamma > 0$. If there exists  $2p-1\leq \alpha\leq n-1-2p$ with $\mbox{gcd}(n,\alpha)=1$ then
there exists a globally simple support shifted Heffter array $A$ where the non-empty cells are precisely on the diagonals $D_i$ for $i\in[4p-1]\cup\{2p+\alpha\}$.
\end{theorem}

The proof of Theorem \ref{main} will be broken into sections. In Subsection \ref{definition} we will define an array $A$ that has $4p$ entries per row and column, with the right support, thus  verifying  that $A$ satisfies Properties P1 and P2. Then in Subsection
\ref{rowcolsums} we will show that each row and column of $A$ sums equal to $0$, thus verifying $A$ satisfies Property P3. Finally in Subsections \ref{partrowsum}, \ref{partcolsum} and \ref{partcolsum0} we will verify that $A$ satisfies Property P4 by showing, respectively, that the row partial sums are distinct, the partial sums for the non-zero columns are distinct and then finally the partial sums for column $0$ are distinct modulo $2(4p+\gamma)n+1$.

\begin{remark}
Throughout Section \ref{thmex} it will be assumed that $p>0$, $n\geq 4p$, and $\gamma> 0$, $2p-1\leq \alpha\leq n-1-2p$ and gcd$(\alpha,n)=1$.  We will define $I=[p]$, $2I=\{2i\mid i\in I\}$, $J=[p-1]$, $\mathbb{D}=\{0,1,\dots,4p-2,2p+\alpha\}$ and $T=\mathbb{D}\setminus 2I.$ Further we remind the reader that row and column numbers will be calculated modulo $n$ with residues from $[n]$, while entries are
calculated as integers.
\end{remark}
\subsection{Definition of the array $A$}\label{definition}
Let $A=[A(i,j)]$ be an $n\times n$ array with filled cells defined by the $4p$ diagonals
$$D_{2i}, D_{2i+1}, D_{2p}, D_{2p+1+2j}, D_{2p+2+2j},D_{2p+\alpha},$$
where $ i\in I$ and $ j\in J$, and with entries for each $x\in [n]$:
\begin{eqnarray}
(\gamma+2)n+4in-2x&\mbox{ in cell}&(2i-x,-x)\in  D_{2i}\nonumber,\nonumber\\
-\gamma n-4in-1-2x &\mbox{ in cell}&(2i+1+x,x)\in D_{2i+1},\nonumber \\
-(4p+\gamma)n+2x&\mbox{ in cell}&(2p-\alpha x,-\alpha x)\in  D_{2p},\nonumber\label{A}\\
(4p+\gamma-6)n-4jn+1+2x&\mbox{ in cell}&(2p+1+2j-x,-x)\in  D_{2p+1+2j},\nonumber\\
-(4p+\gamma-4)n+4jn+2x&\mbox{ in cell}&(2p+2+2j+x,x)\in  D_{2p+2+2j},\nonumber\\
(4p+\gamma-2)n+1+2x&\mbox{ in cell}&(2p+\alpha+\alpha x,\alpha x)\in  D_{2p+\alpha}\nonumber.
\end{eqnarray}
It is useful to note that the set $\mathbb{D}$ contains the indices for the non-empty diagonals of $A$.

 Then for each $i\in I$ and $j\in J$:
\begin{eqnarray}
 s(D_{2i}\cup D_{2i+1})  &=& \{\gamma n+4in+1,\dots,(\gamma+2)n+4in\};\nonumber\\
 s(D_{2p+1+2j}\cup D_{2p+2+2j}) & = & \{(4p+\gamma-6)n-4n+1,\dots,(4p+\gamma-4)n-4jn\}\nonumber \\
& = &  \{(\gamma+2)n+4\epsilon n+1,\dots,(\gamma+4)n+4\epsilon n\}\  (\epsilon=p-2-j);\nonumber\\
 s(D_{2p}\cup D_{2p+\alpha}) & = & \{(4p+\gamma)n-2n+1,\dots,(4p+\gamma)n\}.\nonumber
\end{eqnarray}

Hence $s(A)=\{\gamma n+1,\dots,(4p+\gamma)n\}$.

\begin{example} \label{appie}
Here we display a support shifted Heffter array $H(17;12,3)$ (the array $A$ above) illustrating Theorem $\ref{main}$ with $\alpha=6$.
\begin{scriptsize}
\begin{center}
\begin{tabular}{|c|c|c|c|c|c|c|c|c|c|c|c|c|c|c|c|c|}
\hline
85&&&&&252&&-105&104&-169&168&-253&-212&213&-148&149&-84\\
\hline
-52&53&&&&&224&&-103&102&-167&166&-225&-214&215&-150&151\\
\hline
153&-54&55&&&&&230&&-101&100&-165&164&-231&-216&217&-152\\
\hline
-120&121&-56&57&&&&&236&&-99&98&-163&162&-237&-218&219\\
\hline
221&-122&123&-58&59&&&&&242&&-97&96&-161&160&-243&-220\\
\hline
-188&189&-124&125&-60&61&&&&&248&&-95&94&-159&158&-249\\
\hline
-255&-190&191&-126&127&-62&63&&&&&254&&-93&92&-157&156\\
\hline
154&-227&-192&193&-128&129&-64&65&&&&&226&&-91&90&-155\\
\hline
-187&186&-233&-194&195&-130&131&-66&67&&&&&232&&-89&88\\
\hline
86&-185&184&-239&-196&197&-132&133&-68&69&&&&&238&&-87\\
\hline
-119&118&-183&182&-245&-198&199&-134&135&-70&71&&&&&244&\\
\hline
&-117&116&-181&180&-251&-200&201&-136&137&-72&73&&&&&250\\
\hline
222&&-115&114&-179&178&-223&-202&203&-138&139&-74&75&&&&\\
\hline
&228&&-113&112&-177&176&-229&-204&205&-140&141&-76&77&&&\\
\hline
&&234&&-111&110&-175&174&-235&-206&207&-142&143&-78&79&&\\
\hline
&&&240&&-109&108&-173&172&-241&-208&209&-144&145&-80&81&\\
\hline
&&&&246&&-107&106&-171&170&-247&-210&211&-146&147&-82&83\\
\hline
\end{tabular}
\end{center}
\end{scriptsize}

\end{example}

To prove that the array is globally simple we must verify that all sequential partial sums are distinct. For the above example we give the row and column partial sums in the Appendix, where we have listed the row (column) number and each partial sum beginning  with the entry in diagonal $D_0$.
These partial sums are considered modulo $2(4p+\gamma)n+1=511$ so it is important to check carefully when the absolute value of the partial sums exceeds $255$.

\subsection{Row sums and column sums}\label{rowcolsums}
For a given row $a\in [n]$ and all $i\in I$, there exists $x_1,x_2\in [n]$ such that $a=2i-x_1\md{n}$ and $a=2i+1+x_2\md{n}$. Thus  $x_1+x_2+1=0\md{n}$  and so $x_1+x_2=n-1.$ Consequently for all $ i\in I$,
\begin{eqnarray}
d_{2i}(r_a)+d_{2i+1}(r_a)&=&(\gamma+2)n+4in-2x_1-\gamma n-4in-1-2x_2\nonumber\\
&=&2n-1-2(n-1)=1.\label{RowiSum}
\end{eqnarray}

 Similarly, for a given row $a$ and for all $j\in J$, there exists $  x_1,x_2\in [n]$ such that  $a=2p+1+2j-x_1\md{n}$ and $a=2p+2+2j+x_2\md{n}$, implying $x_1+x_2+1=0\md{n}$, and so $x_1+x_2=n-1$. Consequently for all $j\in J$,
 \begin{eqnarray}
 d_{2p+2j+1}(r_a)+d_{2p+2j+2}(r_a)&=&(4p+\gamma-6-4j)n+1+2x_1-(4p+\gamma-4+4j)n+2x_2\nonumber\\
 &=&-1.\label{RowjSum}
 \end{eqnarray}

  Finally for $ x_1,x_2\in [n]$,  $a=2p-\alpha x_1\md{n}$ and $a=2p+\alpha+\alpha x_2\md{n}$ implies that $\alpha(x_1+x_2+1)=0\md{n}$. Since gcd$(\alpha,n)=1$, $x_1+x_2+1=0\md{n}$ and again $x_1+x_2=n-1$. Hence
  \begin{eqnarray}
d_{2p}(r_a)+d_{2p+\alpha}(r_a)&=&  -(4p+\gamma)n+2x_1+(4p+\gamma-2)n+1+2x_2\nonumber\\
&=&-1.\label{RowAlphaSum}
 \end{eqnarray}

Therefore as required, the sum of the entries in row $a$ of $A$ is $(1\times p) +(-1\times(p-1))-1=0$.

 In column $a=0$ the sum of the entries is
\begin{eqnarray}\label{col0sum}
&&\displaystyle (d_{2p}(c_0)+d_{2p+\alpha}(c_0))+\sum_{i=0}^{p-1} \left(d_{2i}(c_0)+d_{2i+1}(c_0)\right)+\sum_{j=0}^{p-2} \left(d_{2p+2j+1}(c_0)+d_{2p+2j+2}(c_0)\right)\nonumber \\
&=&\displaystyle -(2n-1)+\sum_{i=0}^{p-1} (2n-1)+\sum_{j=0}^{p-2} -(2n-1)=0.
\end{eqnarray}

For a given column $a\neq 0$, there exists $ x_1,x_2\in [n]$ such that $a=-x_1\md{n}$ and $a=x_2\md{n}$ or equivalently $x_1=n-a$ and $x_2=a$. Thus for all $i\in I$ and for all $j\in J$
\begin{eqnarray}
d_{2i}(c_a)+d_{2i+1}(c_a)&=&(\gamma+2)n+4in-2(n-a)-(\gamma n+4in+1)-2a=-1,\label{ColiSum}\\
d_{2p+2j+1}(c_a)+d_{2p+2j+2}(c_a)&=&(4p+\gamma-6-4j)n+1+2(n-a)-(4p+\gamma-4-4j)n+2a\nonumber\\
&=&1.\label{ColjSum}
\end{eqnarray}

 Furthermore setting $a=-\alpha x_1\md{n}$ and $a=\alpha x_2\md{n}$ we see that $0=\alpha(x_1+x_2)\md{n}$. Now since gcd$(\alpha,n)=1$, $x_1+x_2=0\md{n}$ and so $x_1+x_2=n$. Hence
 \begin{eqnarray}\label{csum-2p+alpha}
d_{2p}(c_a)+d_{2p+\alpha}(c_a)= -(4p+\gamma)n+2x_1+(4p+\gamma-2)n+1+2x_2=1.
\end{eqnarray}

Hence, the sum of the entries in column $a\neq 0$ of $A$ is $(-1\times p)+(1\times(p-1))+1=0$.

\subsection{Distinct partial sums for rows}\label{partrowsum}
For a given row $a$ we will calculate  $\Sigma(x)=\sum_{i=0}^x d_i(r_a)$, for each $x\in \mathbb{D}$, and show that $\Sigma(x_1)\neq \Sigma(x_2)\md{2(4p+\gamma)+1}$ for each $x_1,x_2\in \mathbb{D}$ (Note that these sums cover the entries of the non-empty diagonals).

 Recall that Equations \eqref{RowiSum} and \eqref{RowjSum}  give $d_{2i}(r_a)+d_{2i+1}(r_a)=1$ and $d_{2p+2j+1}(r_a)+d_{2p+2j+2}(r_a)=-1$, for all $i\in I$ and for all $j\in J$.
Then using the definition of the array $A$ we may  evaluate and determine bounds for $\Sigma(x)$ as follows.
  \begin{eqnarray*}
        \gamma n+1<& \Sigma(2i)=d_{2i}(r_a)+i&< (4p+\gamma-2)n+p,\\
         0<&\Sigma(2i+1)=i+1       &< p+1,\\
        -(4p+\gamma)n<& \Sigma(2p)=d_{2p}(r_a)+p &< -(4p+\gamma-2)n+p-1,\\
     \Sigma(2p)<&\Sigma(2p+2j+1)=d_{2p}(r_a)+p + d_{2p+2j+1}(r_a)-j&<0,\\
     -(4p+\gamma)n<&\Sigma(2p+2j+2)=d_{2p}(r_a)+   p- (j+1)&<\Sigma(2p),\\
       &\Sigma(2p+\alpha)= 0.&	
  \end{eqnarray*}

Also, for all $i^\prime\in I\setminus\{p-1\}$ and for all $j^\prime\in J\setminus\{p-2\}$, $$\Sigma(2(i^\prime+1))-\Sigma(2i^\prime)=d_{2i^\prime+2}(r_a)+i^\prime+1-d_{2i^\prime}(r_a)-i^\prime$$ $$=4n+1-2(2i^\prime+2-a\md{n})+2(2i^\prime-a\md{n})\geq 4n-3 \mbox{\rm \quad  and}$$
$$\Sigma(2p+2(j^\prime+1)+1)-\Sigma(2p+2j^\prime+1)=d_{2p+2j^\prime+3}(r_a)-(j^\prime+1)-d_{2p+2j^\prime+1}(r_a)+j^\prime$$ $$=-4n-1+2(2p+2j^\prime+3-a\md{n})-2(2p+2j^\prime+1-a\md{n})\leq -4n+3.$$ Thus  for $i\in I$ and $j\in J$, the function $f(i)=\Sigma(2i)$ is  strictly increasing and the function $g(j)=\Sigma(2p+2j+1)$ is strictly decreasing.

Hence\begin{eqnarray}
&&\Sigma(4p-2)<\Sigma(4p-4)<\dots<\Sigma(2p+2)<\Sigma(2p)<-(4p+\gamma-3)n, \nonumber\\
&&\Sigma(2p)<\Sigma(4p-3)<\Sigma(4p-5)<\dots<\Sigma(2p+1)<0,\label{Rowparsum4p}\\
&&0< \Sigma(1) <\Sigma(3)<\dots<\Sigma(2p-1)<p+1 <\gamma n < \Sigma(0)<\Sigma(2)<\dots<\Sigma(2p-2).  \nonumber
  \end{eqnarray}

 Furthermore, for all $x\in \mathbb{D}$, the row partial sums $|\Sigma(x)|\leq (4p+\gamma)n$, and  so by Remark (\ref{mods}) $\Sigma(x)\equiv \Sigma(y)\md{2(4p+\gamma)n+1}$ if and only if $\Sigma(x)=\Sigma(y)$. Hence  for all $i\in I$ and  for all $j\in J$ the partial sums calculated on
row $a$ are all distinct modulo $2(4p+\gamma)n+1$.

\subsection{Distinct partial sums for non-zero columns}\label{partcolsum}

Similarly to above, we calculate  $\overline{\Sigma}(x)=\sum_{i=0}^x d_i(c_a)$, for each $ x\in \mathbb{D}$ and show that $\overline{\Sigma}(x_1)\neq \overline{\Sigma}(x_2)\md{2(4p+\gamma)+1}$ for each $x_1,x_2\in \mathbb{D}$.

Equations \eqref{ColiSum} and \eqref{ColjSum} imply that $d_{2i}(c_a)+d_{2i+1}(c_a)=-1$ and $d_{2p+2j+1}(c_a)+d_{2p+2j+2}(c_a)=1$, for all $i\in I$ and for all $j\in J$. Then using the definition of the array $A$ we may evaluate and determine bounds for $\overline{\Sigma}(x)$ as follows.
 \begin{eqnarray*}
     \gamma n<&\overline{\Sigma}(2i)=d_{2i}(c_a)-i&<(4p+\gamma-2)n, \\
   -p\leq&  \overline{\Sigma}(2i+1)=-(i+1)&< 0,\\
 -(4p+\gamma)n-p\leq& \overline{\Sigma}(2p)=d_{2p}(c_a)-p &\leq-(4p+\gamma-2)n-p-2,\\
 -(4p-2)n<& \overline{\Sigma}(2p+2j+1)=&\\
 &d_{2p}(c_a)-p+d_{2p+2j+1}(c_a)+j&<-2n-p,\\
  \overline{\Sigma}(2p)<&\overline{\Sigma}(2p+2j+2)=d_{2p}(c_a)-p+(j+1)&<-(4p+\gamma-2)n,\\
  &\overline{\Sigma}(2p+\alpha)= 0.&	
  \end{eqnarray*}
Furthermore, for all $i^\prime\in I\setminus\{p-1\}$ and for all $j^\prime\in J\setminus\{p-2\}$, $$\overline{\Sigma}(2(i^\prime+1))-\overline{\Sigma}(2i^\prime)=d_{2i^\prime+2}(c_a)-(i^\prime+1)-d_{2i^\prime}(c_a)+i^\prime=
4n-1\mbox{\rm \quad  and}$$  $$\overline{\Sigma}(2p+2(j^\prime+1)+1)-\overline{\Sigma}(2p+2j^\prime+1)=d_{2p+2j^\prime+3}(c_a)+j^\prime+1-d_{2p+2j^\prime+1}(c_a)-j^\prime=-4n+1.$$ Thus for $i\in I$ and $j\in J$ the function $f(i)=\overline{\Sigma}(2i)$ is  strictly increasing  and $g(j)=\overline{\Sigma}(2p+2j+1)$ is strictly decreasing.

  Hence
  \begin{eqnarray}
  &&-(4p+\gamma+1)n<\overline{\Sigma}(2p)<\overline{\Sigma}(2p+2)<\dots<\overline{\Sigma}(4p-2)<\overline{\Sigma}(4p-3)\nonumber\\
  &&\overline{\Sigma}(4p-3) <\overline{\Sigma}(4p-5)<\dots<\overline{\Sigma}(2p+1)<-n<\nonumber\\
  &&\overline{\Sigma}(2p-1)<\dots<\overline{\Sigma}(3)<\overline{\Sigma}(1)<0<\gamma n < \overline{\Sigma}(0)<\overline{\Sigma}(2)<\dots<\overline{\Sigma}(2p-2).\label{Colparsum4p}
  \end{eqnarray}

 Thus for column $a\neq 0$ and  each $x\in \mathbb{D}$, the partial  sum   $|\overline{\Sigma}(x)|< (4p+\gamma+1)n$. Further for $x\in \mathbb{D}\setminus F$, where $F=\{2p,2p+2,\dots,4p-2\}$, $|\overline{\Sigma}(x)|\leq(4p+\gamma-2)n<(4p+\gamma)n$. Thus
 for all $x,y\in \mathbb{D}\setminus F$, $\overline{\Sigma}(x)\equiv \overline{\Sigma}(y)\md{2(4p+\gamma)n+1}$ if and only if $\overline{\Sigma}(x)=\overline{\Sigma}(y)$ by (\ref{mods}).
 Furthermore, for all $x\in F$, $2(4p+\gamma)n+1+\overline{\Sigma}(x)>(4p+\gamma)n-p+1>(4p+\gamma-2)n>\overline{\Sigma}(y)$ for all $y\in \mathbb{D}$. Hence, in column $a\neq 0$   the partial  sums calculated on
diagonals $D_{x}$, $x\in {\mathbb D}$, are distinct modulo $2(4p+\gamma)n+1$.

\subsection{Distinct partial sums for column zero}\label{partcolsum0}

From Section \ref{definition}, for $i\in I$ and $j\in J$ the entries in column $0$ are
\begin{eqnarray*}
d_{2i}(c_0)			&=&			\gamma n+4in+2n,\\
d_{2i+1}(c_0)		&=&			-\gamma n-4in-1,\\
d_{2p}(c_0)			&=&			-4pn-\gamma n,\\
d_{2p+1+2j}(c_0)	&=&			4pn+\gamma n-6n-4jn+1,\\
d_{2p+2+2j}(c_0)	&=&			-4pn-\gamma n+4n+4jn,\\
d_{2p+\alpha}(c_0)	&=&			4pn+\gamma n-2n+1,
\end{eqnarray*}
 and thus $d_{2i}(c_0)+d_{2i+1}(c_0)=2n-1$ and $d_{2p+2j+1}(c_0)+d_{2p+2j+2}(c_0)=-(2n-1)$. Thus, for $i\in I$ and $j\in J$,  the partial sums may be calculated and bounded as follows.
 \begin{eqnarray*}
  \overline{\Sigma}(2i)&=&(\gamma+2+4i)n+(2n-1)i=(\gamma+2+6i)n-i\\
  &\Rightarrow&0< \overline{\Sigma}(0)< \overline{\Sigma}(2)<\dots<\overline{\Sigma}(2p-2)<2(4p+\gamma)n+1,\\
  \overline{\Sigma}(2i+1)&=&(2n-1)(i+1)=2(i+1)n-(i+1)\\
  &\Rightarrow&0<2n-1=\overline{\Sigma}(1)<\overline{\Sigma}(3)<\dots <\overline{\Sigma}(2p-1)=(2n-1)p,\\
  \overline{\Sigma}(2p)&=&(2n-1)p-(4p+\gamma)n=-(2p+\gamma)n-p\\
  &\Rightarrow& \overline{\Sigma}(2p)<0,  \\
 \overline{\Sigma}(2p+2j+1)&=&(p-j)(2n-1)-(4p+\gamma)n +(4p+\gamma-6)n-4jn+1\\
  &=&(2p-6j-6)n-(p-j)+1\\
   &\Rightarrow& -(4p-6)n-1=\overline{\Sigma}(4p-3)<\dots< \overline{\Sigma}(2p+1) <2pn,\\
\overline{\Sigma}(2p+2j+2)&=& (p-j-1)(2n-1)-(4p+\gamma)n=  -(2p+2j+\gamma+2)n-(p-j-1) \\
&\Rightarrow&-(4p+\gamma-2)n-1= \overline{\Sigma}(4p-2)<\dots<\overline{\Sigma}(2p+2)<\overline{\Sigma}(2p),\\
  \overline{\Sigma}(2p+\alpha)&=&0.
  \end{eqnarray*}

Note that for all $x\in \mathbb{D}$ and $y\in \mathbb{D}\setminus 2I$:
\begin{eqnarray}
&& \overline{\Sigma}(x)=\beta n+\epsilon\mbox{\rm \ for some\ } \beta,\epsilon\in \mathbb{Z}\mbox{\ with\ }-\frac{n}{4}\leq-p\leq\epsilon\leq p\leq \frac{n}{4},\label{star4}\\
&& |\overline{\Sigma}(x)|\leq 2(4p+\gamma)n+1, \label{star1}\\
&& |\overline{\Sigma}(y)|\leq (4p+\gamma)n. \label{star3a}
\end{eqnarray}

We will proceed by checking a number of cases individually. In what follows we will make extensive use of (\ref{star4}) and (\ref{n}).

For all $ i_1,i_2,i\in I$ and $ j_1,j_2,j\in J$:
\begin{itemize}

\item[1(i)] Suppose that $\overline{\Sigma}(2i_1)=\overline{\Sigma}(2i_2+1)\md{2(4p+\gamma )n+1}$. Then  (\ref{modl}) and (\ref{star1}) implies
    $(2+\gamma+6i_1)n-i_1= 2(i_2+1)n-(i_2+1)$. Hence $i_1=i_2+1$ and $2+\gamma+6i_1=2i_2+2$. But then $\gamma=-4i_2-6$. This  contradicts $\gamma>0$.
\item[1(ii)] Suppose that $\overline{\Sigma}(2i)=\overline{\Sigma}(2p)\md{2(4p+\gamma )n+1}$. Then by (\ref{mod=})
\begin{eqnarray*}
(2+\gamma+6i)n-i&=&2(4p+\gamma)n+1-(2p+\gamma)n-p\\
\Rightarrow \quad (2+\gamma+6i)n-i&=&(6p+\gamma)n+1-p.
\end{eqnarray*} This implies $i=p-1$ and $2+6i=6p$ leading to the contradiction $-4=0.$

\item[1(iii)] Suppose that $\overline{\Sigma}(2i)=\overline{\Sigma}(2p+2j+1)\md{2(4p+\gamma )n+1}$. Then
\begin{eqnarray*}
(2+\gamma+6i)n-i&=& (2p-6j-6)n-(p-j)+1 \mbox{ or} \\
- 2(4p+\gamma)n-1+(2+\gamma+6i)n-i&=& (2p-6j-6)n-(p-j)+1.
 \end{eqnarray*} The former case implies $i=p-j-1$ and so $2+\gamma+6p-6j-6=2p-6j-6$ but then $\gamma=-4p-2$ which contradicts $\gamma\geq 0$. In the latter case we have $(-8p-\gamma+6i+2)n-(i+1)=(2p-6j-6)n-(p-j)+1$, which implies $p-2=i+j$ and so $\gamma=-10p+6(i+j)+8=-10p+6p-12+8=-4p-4$,  a contradiction.

\item[1(iv)] Suppose that $\overline{\Sigma}(2i)=\overline{\Sigma}(2p+2j+2)\md{2(4p+\gamma )n+1}$. Then by (\ref{mod=})
\begin{eqnarray*}
(2+\gamma+6i)n-i&=& 2(4p+\gamma)n+1 -(2p+2j+\gamma+2)n-(p-j-1)\\
\Rightarrow \quad (2+\gamma+6i)n-i&=&(6p+\gamma-2j-2)n-(p-j-2).
\end{eqnarray*}
 This implies  $i=p-j-2$ and $2+6i=6p-2j-2$ or equivalently $2+6(p-j-2)=6p-6j-2$ and so  $-10=-2$, a contradiction.
 \\ \\
 \noindent  For the remaining cases we will use (\ref{mods}) together with (\ref{star3a}).
\\
\item[2(i)] $\overline{\Sigma}(2i+1)>0$ and $\overline{\Sigma}(2p),\overline{\Sigma}(2p+2j+2)<0$ and so  we have $\overline{\Sigma}(2i+1)\not\equiv\overline{\Sigma}(2p)$ $\md{2(4p+\gamma )n+1}$ and $\overline{\Sigma}(2i+1)\not\equiv \overline{\Sigma}(2p+2j+2)\md{2(4p+\gamma )n+1}$.

\item[2(ii)]
 Suppose that $\overline{\Sigma}(2i+1)\equiv \overline{\Sigma}(2p+2j+1)\md{2(4p+\gamma )n+1}$.  Then we have
 $(2p-6j-6)n-(p-j)+1=(2i+2)n-(i+1)$. So $p=j+i+2$ and $2p-6j-2i-8=0$ which implies $2(i+j+2)-6j-2i-8=0$ and then $j=-1$, a contradiction.

\item[3(i)]  Suppose that $\overline{\Sigma}(2p)\equiv \overline{\Sigma}(2p+2j+1)\md{2(4p+\gamma )n+1}$. Then
$-(2p+\gamma)n-p= (2p-6j-6)n-(p-j)+1$. So $-p=-p+j+1$ but then $j=-1$.

\item[3(ii)] Suppose that $\overline{\Sigma}(2p)\equiv \overline{\Sigma}(2p+2j+2)\md{2(4p+\gamma )n+1}$. Then
$-(2p+\gamma)n-p= -(2p+2j+\gamma+2)n-(p-j-1)$. So $p=p-j-1$ but then  $j=-1$.


\item[4(i)] Suppose that $\overline{\Sigma}(2p+2j_1+1)\equiv \overline{\Sigma}(2p+2j_2+2)\md{2(4p+\gamma )n+1}$. Then
    $(2p-6j_1-6)n-(p-j_1)+1=-(2p+2j_2+\gamma+2)n-(p-j_2-1)$. So $-(p-j_1)+1=-p+j_2+1$ and $4p-6j_1+2j_2-4+\gamma=0$. Then $j_1=j_2$ and $\gamma=-4(p-j_1-1)<0$.

\item[4(ii)] Suppose that $\overline{\Sigma}(2p+2j+1)\equiv 0\md{2(4p+\gamma)n+1}$ then $-p+j+1=0$ which implies $j=p-1>p-2$, a contradiction.

    \end{itemize}

Hence for column $0$ all the partial sums are distinct modulo $2(4p+\gamma)n+1$. This proves Theorem \ref{main}.

\section{Globally simple Heffter arrays $H(n;4p+3)$}

In this section we will merge a Heffter array $H(n;3)$ with the support shifted Heffter array $H(n;4p,3)$ given by Theorem
 \ref{main}
 to obtain a globally simple Heffter array $H(n;4p+3)$. First we need a suitable $H(n;3)$.

\begin{theorem} {\rm \cite{ADDY}} Let $n\equiv 0,1\md{4}$. Then there exists a Heffter array $H(n;3)$ that has the following properties: non-empty cells are only on diagonals $D_0$, $D_{n-1}$ and $D_{1}$; $s_{L'}(D_0)=\{1,\dots,n\}$; entries of $L'$ on $D_{n-1}$ are all positive and entries of $L'$ on $D_{1}$ are all negative.
\label{addict}
\end{theorem}

\begin{theorem} \label{Ladder}
Let $n\equiv 1\md{4}$. Then for each $0\leq \beta\leq n-5$ there exists a Heffter array $H(n;3)$, denoted by $L$, with the following properties
\begin{itemize}
\item The non empty cells are exactly on the diagonals $D_{\beta}$, $D_{\beta+2}$ and $D_{\beta+4}$,
\item $s(D_{\beta+2})=\{1,\dots,n\}$,
\item $s(D_\beta\cup D_{\beta+4})=\{n+1,\dots,3n\}$,
\item entries on $D_{\beta}$ are all positive,
\item entries on $D_{\beta+4}$ are all negative,
\item the array defined by $M=[M(i,j)]$ where $M(i,j)=L(i+1,j+1)$, $i,j\in [n]$ retains the above properties.
\end{itemize}
\end{theorem}

\begin{proof}
Let $L'$ be a Heffter array $H(n;3)$ with the properties from Theorem \ref{addict} where $n\equiv 1\md{4}$.
  Now define $L(2(i+1)+\beta,2j)=L'(i,j)$ for all $i,j\in [n]$ where operations on coordinates are taken modulo $n$. As $n$ is odd, for any given $a,b\in [n]$ there exists unique $i,j\in [n]$ such that $2(i+1)+\beta\equiv a\md{n}$ and $2j\equiv b\md{n}$. Hence we may obtain $L$ by applying row and column permutations to $L'$. Therefore $L$ is also a Heffter array $H(n;3)$. Furthermore, the entries on $D_\beta$ of $L$ are exactly the entries on $D_{n-1}$ of $L'$; consequently
entries on $D_\beta$ of $L$ are all positive. Also the set of entries on $D_{\beta+2}$ of $L$ are exactly the set of entries on $D_{0}$ of $L'$; consequently $s_{L}(D_{\beta+2})=\{1,\dots,n\}$. Similarly the set of entries on $D_{\beta+4}$ of $L$ are exactly the set of entries on $D_{1}$ of $L'$; consequently entries on $D_\beta+4$ of $L$ are all negative.

Finally, it is clear that the array $M$ retains the above properties, since each diagonal retains the same set of symbols under this transformation.
\end{proof}
\begin{figure}
\begin{footnotesize}
\begin{tabular}{cc}
\begin{tabular}{|c|c|c|c|c|c|c|c|c|}
\hline
-8 &18 &   &   &   &   &   &  &-10\\\hline
-19&-7 &26 &   &   &   &   &  & \\\hline
   &-11&-6 &17 &   &   &   &  &\\\hline
   &   &-20&-5 &25 &   &   &  &\\\hline
   &   &   &-12&-9 &21 &   &  &\\\hline
   &   &   &   &-16&3  &13 &  &\\\hline
   &   &   &   &   &-24&2  &22&\\\hline
   &   &   &   &   &   &-15&1 &14\\\hline
27 &   &   &   &   &   &   &-23&-4\\ \hline
\end {tabular}
&
\begin{tabular}{|c|c|c|c|c|c|c|c|c|}
\hline
   &   &   &   &-20&   &-5 &  &25\\\hline
27 &   &   &   &   &-23&   &-4& \\\hline
   &21 &   &   &   &   &-12&  &-9\\\hline
-8 &   &18 &   &   &   &   & -10 &\\\hline
   &3  &   &13 &   &   &   &  &-16\\\hline
-19&   &-7 &   &26 &   &   &  &\\\hline
   &-24&   &2  &   &22 &   &  &\\\hline
   &   &-11&   &-6 &   &17 &  & \\\hline
   &   &   &-15&   &1  &   &14&\\\hline
\end{tabular}
\end{tabular}
\caption{The arrays $L$ and $L'$ from Theorem \ref{Ladder} with $\beta=1$.}
\end{footnotesize}
\end{figure}

By similar reasoning to the previous theorem, we also have the following.

\begin{corollary} \label{Ladder2}
Let $n\equiv 0\md{4}$ and gcd$(n,\epsilon)=1$. Then for each $0\leq \beta\leq n-2\epsilon-1$ there exists a Heffter array $H(n;3)$, denoted by $L$, with the following properties
\begin{itemize}
\item The non empty cells are exactly on the diagonals $D_{\beta}$, $D_{\beta+\epsilon}$ and $D_{\beta+2\epsilon}$,
\item $s(D_{\beta+2\epsilon})=\{1,\dots,n\}$ and $s(D_\beta\cup D_{\beta+4\epsilon})=\{n+1,\dots,3n\}$,
\item entries on $D_{\beta}$ are all positive,
\item entries on $D_{\beta+2\epsilon}$ are all negative,
\item the array defined by $M=[M(i,j)]$ where $M(i,j)=L(i+1,j+1)$, $i,j\in [n]$ retains the above properties.
\end{itemize}
\end{corollary}

\subsection{Globally simple $H(n;4p+3)$ when $n\equiv 1\md{4}$}

\begin{theorem}\label{4p+3}
Let $n\equiv 1\md{4}$, $p>0$ and $n\geq 4p+3$. Let $\alpha$ be an integer such that $2p+2\leq \alpha\leq n-2-2p$ and $\mbox{gcd}(n,\alpha)=1$.
Let $\beta=2p+\alpha-3$ and let $L$ be a Heffter array $H(n;3)$ based on $\beta$ satisfying the properties of Theorem
$\ref{Ladder}$. Then
the union of arrays $L$ and the support shifted Heffter array $H(n;4p,3)$ (given by Theorem $\ref{main}$) is a globally simple Heffter array $H(n;4p+3)$ where the entries are on the set of diagonals $D_i$ such that $i$ is in
${\mathbb D}=\{0,1,\dots ,4p-2,2p+\alpha\}\cup \{2p+\alpha-3,2p+\alpha-1,2p+\alpha+1\}$.
\end{theorem}

\begin{proof}
First we will assume that there exists $\alpha$ coprime to $n$ such that $2p+2\leq \alpha\leq n-2-2p$. Next, construct the array $A$ as in Theorem \ref{main} with $\gamma=3$. Then we will merge this array with $L$ as constructed in Theorem \ref{Ladder} with $\beta=2p+\alpha-3$ to get a Heffter array $H(n;4p+3)$ that will be globally simple, which we denote by $B$. Note that since $n\equiv 1\md{4}$, such an $L$ exists.

Define

$B(i,j)= A(i,j)$ if $i-j\not\in\{2p+\alpha-3,2p+\alpha-1,2p+\alpha+1\},$

$B(i,j)= L(i,j)$ if $i-j\in\{2p+\alpha-3,2p+\alpha-1,2p+\alpha+1\}.$

Hence we are positioning diagonals $D_{2p+\alpha-3}$, $D_{2p+\alpha-1}$ and $D_{2p+\alpha+1}$ of $L$ to the empty diagonals $D_{2p+\alpha-3}$, $D_{2p+\alpha-1}$ and $D_{2p+\alpha+1}$ of $A$.
\begin{example}
The array $B$ (a Heffter array $H(17;15)$) when $n=17$, $p=3$ and $\alpha=8$.
\begin{scriptsize}
\begin{center}
\begin{tabular}{|c|c|c|c|c|c|c|c|c|c|c|c|c|c|c|c|c|}
\hline
85&&\textbf{-33}&244&\textbf{13}&&\textbf{20}&-105&104&-169&168&-245&-212&213&-148&149&-84\\
\hline
-52&53&&\textbf{-24}&240&\textbf{-4}&&\textbf{28}&-103&102&-167&166&-241&-214&215&-150&151\\
\hline
153&-54&55&&\textbf{-49}&236&\textbf{12}&&\textbf{37}&-101&100&-165&164&-237&-216&217&-152\\
\hline
-120&121&-56&57&&\textbf{-41}&232&\textbf{-3}&&\textbf{44}&-99&98&-163&162&-233&-218&219\\
\hline
221&-122&123&-58&59&&\textbf{-32}&228&\textbf{11}&&\textbf{21}&-97&96&-161&160&-229&-220\\
\hline
-188&189&-124&125&-60&61&&\textbf{-25}&224&\textbf{-2}&&\textbf{27}&-95&94&-159&158&-225\\
\hline
-255&-190&191&-126&127&-62&63&&\textbf{-48}&254&\textbf{10}&&\textbf{38}&-93&92&-157&156\\
\hline
154&-251&-192&193&-128&129&-64&65&&\textbf{-42}&250&\textbf{-1}&&\textbf{43}&-91&90&-155\\
\hline
-187&186&-247&-194&195&-130&131&-66&67&&\textbf{-31}&246&\textbf{9}&&\textbf{22}&-89&88\\
\hline
86&-185&184&-243&-196&197&-132&133&-68&69&&\textbf{-26}&242&\textbf{8}&&\textbf{26}&-87\\
\hline
-119&118&-183&182&-239&-198&199&-134&135&-70&71&&\textbf{-47}&238&\textbf{17}&&\textbf{30}\\
\hline
\textbf{35}&-117&116&-181&180&-235&-200&201&-136&137&-72&73&&\textbf{-51}&234&\textbf{16}&\\
\hline
&\textbf{46}&-115&114&-179&178&-231&-202&203&-138&139&-74&75&&\textbf{-39}&230&\textbf{-7}\\
\hline
\textbf{15}&&\textbf{19}&-113&112&-177&176&-227&-204&205&-140&141&-76&77&&\textbf{-34}&226\\
\hline
222&\textbf{-6}&&\textbf{29}&-111&110&-175&174&-223&-206&207&-142&143&-78&79&&\textbf{-23}\\
\hline
\textbf{-50}&252&\textbf{14}&&\textbf{36}&-109&108&-173&172&-253&-208&209&-144&145&-80&81&\\
\hline
&\textbf{-40}&248&\textbf{-5}&&\textbf{45}&-107&106&-171&170&-249&-210&211&-146&147&-82&83\\
\hline
\end{tabular}
\end{center}
\end{scriptsize}
\end{example}
We know that $s(A)=\{\gamma n+1,\dots, (4p+\gamma)n\}$, $s(L)=\{1,\dots,3n\}$ and $\gamma=3$ so $s(B)=\{1,\dots,(4p+3)n\}$. Also as row and column sums of both $A$ and $L$ are $0$, it is easy to see that the row and column sums of $B$ are $0$.

Now we just need to show that row partial sums and column partial sums of $B$ are distinct.

We will use the notation $\Sigma_A(x)$ and $\overline{\Sigma}_A(x)$ to denote the partial sum in the array $A$ as given in the previous section and; $\Sigma_{B}(x)$ and $\overline{\Sigma}_{B}(x)$ to denote the partial sum in the array $B$ as constructed here.
Firstly, $\Sigma_B(i)=\Sigma_A(i)$ and $\overline{\Sigma}_B(i)=\overline{\Sigma}_A(i)$  for all $1\leq i\leq 4p-2$ for all rows and columns of $B$ so row partial sums and column partial sums are distinct modulo $2(4p+3)n+1$ from diagonal $0$ to $4p-2$.

Consider row $a$.
First note that $L(a,a-2p-\alpha+3)+L(a,a-2p-\alpha+1)+L(a,a-2p-\alpha-1)=0$, $L(a,a-2p-\alpha+3)>0$ and
$L(a,a-2p-\alpha-1)<0$.
It was shown in Section \ref{partrowsum} that $\Sigma_B(4p-2)=d_{2p}(r_a)+1$.  Hence we have
\begin{eqnarray}
\Sigma_B(2p+\alpha-3)&=&d_{2p}(r_a)+1+L(a,a-2p-\alpha+3)<0,\nonumber\\
\Sigma_B(2p+\alpha-1)&=&d_{2p}(r_a)+1-L(a,a-2p-\alpha-1)<0,\nonumber\\
\Sigma_B(2p+\alpha)&=&-L(a,a-2p-\alpha-1)>0,\nonumber\\
\Sigma_B(2p+\alpha+1)&=&0.\nonumber
\end{eqnarray}
By Theorem \ref{Ladder}, it follows that \begin{eqnarray} n+1\leq\Sigma_B(2p+\alpha)\leq 3n \label{RangeL}\end{eqnarray}
 so by the inequality (\ref{Rowparsum4p}), $\Sigma_B(2p+\alpha)\neq \Sigma_B(i)$ for each $0\leq i\leq 4p-2$.

Now, from the definition of the array $A$, any entry in diagonal $D_{4p-3}$ is greater than $5n$. Thus, from Section 3.3 $\Sigma_B(4p-3)=d_{2p}(r_a)+d_{4p-3}(r_a)+2>d_{2p}(r_a)+5n$.
 So, \begin{eqnarray*}
\Sigma_B(2p) & = & d_{2p}(r_a)+p <d_{2p}(r_a)+n+2 \\
& \leq  & \Sigma_B(2p+\alpha-3), \Sigma_B(2p+\alpha-1)\leq d_{2p}(r_a)+1+3n< \Sigma_B(4p-3)
\end{eqnarray*} so by inequality (\ref{Rowparsum4p}) $\Sigma_B(2p+\alpha-3)\neq \Sigma_B(i)$ and $\Sigma_B(2p+\alpha-1)\neq \Sigma_B(i)$ for all $0\leq i\leq 4p-2$.

Next consider column $a\neq 0$. We have $L(2p+a+\alpha-3,a)+L(2p+a+\alpha-1,a)+L(2p+a+\alpha+1,a)=0$, with the first of these terms positive and the final term negative. Also it was shown in Section \ref{partcolsum} that $\overline{\Sigma}_B(4p-2)=d_{2p}(c_a)-1$. So
 \begin{eqnarray}
\overline{\Sigma}_B(2p+\alpha-3)&=&d_{2p}(c_a)-1+L(2p+a+\alpha-3,a)<0,\nonumber\\
\overline{\Sigma}_B(2p+\alpha-1)&=&d_{2p}(c_a)-1-L(2p+a+\alpha+1,a)<0,\nonumber\\
\overline{\Sigma}_B(2p+\alpha)&=&-L(2p+a+\alpha+1,a)>0,\nonumber\\
\overline{\Sigma}_B(2p+\alpha+1)&=&0.\nonumber
\end{eqnarray}

Similarly to (\ref{RangeL}),  \begin{eqnarray} n+1\leq\overline{\Sigma}_B(2p+\alpha)\leq 3n. \label{RangeLc}\end{eqnarray}
Thus for each column $a\neq 0$ and $1\leq i\leq 4p-2$ we have $\overline{\Sigma}_B(2p+\alpha)\neq \overline{\Sigma}_B(i)$ by inequalities (\ref{RangeLc}) and  (\ref{Colparsum4p}).

Also
\begin{eqnarray*}
\overline{\Sigma}_B(4p-2) & = & d_{2p}(c_a)-1 <d_{2p}(c_a)+n \\
& \leq & \overline{\Sigma}_B(2p+\alpha-3), \overline{\Sigma}_B(2p+\alpha-1)\leq d_{2p}(c_a)-1+3n<\overline{\Sigma}_B(4p-3)
\end{eqnarray*}
 so by inequality (\ref{Colparsum4p}) $\overline{\Sigma}_B(2p+\alpha-3)\neq \overline{\Sigma}_B(i)$ and $\overline{\Sigma}_B(2p+\alpha-1)\neq \overline{\Sigma}_B(i)$
 for all $0\leq i\leq 4p-1$.

Finally consider column $0$.

 By Theorem \ref{Ladder}, the Heffter array $L$ may be replaced by
 the array $M=[M(i,j)]$ where $M(i,j)=L(i+1,j+1)$, $i,j\in [n]$, while retaining the properties we have so far required. In effect, we may thus apply this transformation to the diagonals $D_{2p+\alpha-3}$, $D_{2p+\alpha-1}$ and $D_{2p+\alpha+1}$ without changing the rest of the array, and without changing the validity of the above arguments.
Since $n\geq 9$, we may thus assume that
 $\{L(2p+\alpha-3,0),-L(2p+\alpha+1,0)\}\cap\{2n-1,2n-(2p+1)/3\}=\emptyset$.
 By Section \ref{partcolsum0} we have:
\begin{eqnarray*}
\overline{\Sigma}_B(2i)&=&(5+6i)n-i,\nonumber\\
 \overline{\Sigma}_B(2i+1)&=&(2n-1)(i+1)=2(i+1)n-(i+1),\nonumber\\
 \overline{\Sigma}_B(2p)&=&-(2p+3)n-p<0,\nonumber\\
 \overline{\Sigma}_B(2p+2j+1)&=&(2p-6j-6)n-(p-j)+1,\nonumber\\
\overline{\Sigma}_B(2p+2j+2)&=& -(2p+2j+5)n-(p-j-1)<0, \nonumber\\
\overline{\Sigma}_B(2p+\alpha-3)&=&-(4p+1)n-1+L(2p+\alpha-3,0)<0,\nonumber\\
\overline{\Sigma}_B(2p+\alpha-1)&=&-(4p+1)n-1-L(2p+\alpha+1,0)<0,\nonumber\\
\overline{\Sigma}_B(2p+\alpha)&=&-L(2p+\alpha+1,0)>0,\nonumber\\
\overline{\Sigma}_B(2p+\alpha+1)&=&0\nonumber.
\end{eqnarray*}

\begin{itemize}
\item[1(i)]  $\overline{\Sigma}_B(2i)=(5+6i)n-i\geq 5n$ so by (\ref{star1}), (\ref{RangeLc}) and (\ref{modl}), $\overline{\Sigma}_B(2p+\alpha)\neq \overline{\Sigma}_B(2i)$ for all $i\in I$.
\item[1(ii)]  $\overline{\Sigma}_B(2i+1)=(2n-1)(i+1)$ then $\overline{\Sigma}_B(2p+\alpha)\neq \overline{\Sigma}_B(2i+1)\md{2(4p+3)n+1}$ by (\ref{star3a}),  (\ref{RangeLc}) and (\ref{mods}) as $-L(2p+\alpha+1,0)\neq 2n-1$.
\item[1(iii)]  $\overline{\Sigma}(2p),\overline{\Sigma}(2p+2j+2)<0$ hence  $\overline{\Sigma}(2p+\alpha)\neq \overline{\Sigma}(2p)\md{2(4p+3)n+1}$, $\overline{\Sigma}(2p+\alpha)\neq \overline{\Sigma}(2p+2j+2)\md{2(4p+3)n+1}$ for all $j\in J$ by (\ref{star3a}), (\ref{RangeLc}) and (\ref{star3a}).
\item[1(iv)]   Suppose that $\overline{\Sigma}(2p+2j+1)=\overline{\Sigma}(2p+\alpha)\md{2(4p+3)n+1}$. Then by (\ref{mods}) and (\ref{star3a}) $$(2p-6j-6)n-(p-j)+1=-L(2p+\alpha+1,0).$$ Now by (\ref{n}), ((\ref{star4}) and \ref{RangeLc}) $2p-6j-6=2$ which implies $p-4=3j$ hence $j=(p-4)/3$ and $-L(2p+\alpha+1,0)=2n-(2p+1)/3.$

\item[2(i)]  Suppose that $\overline{\Sigma}(2i)=\overline{\Sigma}(2p+\alpha-3)\md{2(4p+3)n+1}$ for some $i\in I$. Then
$\overline{\Sigma}(2i)=(5+6i)n-i=2(4p+3)n+1 -(4p+1)n-1+L(2p+\alpha-3,0)=2(4p+3)n+1+\overline{\Sigma}_B(2p+\alpha-3)$ hence
$ -(4p-6i)n-i=L(2p+\alpha-3,0)$ which implies $-(4p-6i)=2$ and so $L(2p+\alpha-3,0)=2n-(2p+1)/3.$

\item[2(ii)]  $\overline{\Sigma}(2i+1)>0$ for all $i\in [p]$ so $\overline{\Sigma}(2p+\alpha-1)\neq \Sigma(2i+1)\md{2(4p+3)n+1}$ and $\overline{\Sigma}(2p+\alpha-3)\neq \overline{\Sigma}(2i+1)\md{2(4p+3)n+1}$.

\item[2(iii)] Suppose that $\overline{\Sigma}(2p)\equiv\overline{\Sigma}(2p+\alpha-3)\md{2(4p+3)n+1}$ then  $-(4p+1)n-1+L(2p+\alpha-3,0)=-(2p+3)n-p$ so $L(2p+\alpha-3,0)=(2p-2)n-p+1$. Then $2p-2=2$ hence $p=2$ and $L(2p+\alpha-3,0)=2n-1$.

\item[2(iv)] Suppose that $\overline{\Sigma}(2p+2j+1)\equiv\overline{\Sigma}(2p+\alpha-3)\md{2(4p+3)n+1}$ for some
$j\in J$. Then by (\ref{star1}) and (\ref{star3a}), \begin{eqnarray*}
    (2p-6j-6)n-(p-j)+1&=&-(4p+1)n-1+L_n(2p+\alpha-3,0)\\
\Rightarrow \quad (6p-6j-5)n&=&p-j-2+L_n(2p+\alpha-3,0)\leq (p-2)+3n<4n
    \end{eqnarray*}
     and also by (\ref{RangeLc}) $n+1\leq p-j-2+L_n(2p+\alpha-3,0)$ so $1< 6p-6j-5<4$. Hence $6< 6(p-j)<9$, a contradiction.

\item[2(v)]  Suppose that $\overline{\Sigma}(2p+2j+2)\equiv \overline{\Sigma}(2p+\alpha-3)\md{2(4p+3)n+1}$ for some
$j\in J$. Then
\begin{eqnarray*}
    -(2p+2j+5)n-(p-j-1)&=&-(4p+1)n-1+L_n(2p+\alpha-3,0)\\
\Rightarrow \quad     (2p-2j-4)n&=&p-j-2+L(2p+\alpha-3,0)\leq (p-2)+3n<4n
      \end{eqnarray*}
      and also $n+1\leq p-j-2+L(2p+\alpha-3,0)$ so  $1< 2(p-j-2)<4$. Hence $1=p-j-2$ then $j=p-3$ which implies $L(2p+\alpha-3,0)=2n-1$.

\end{itemize}
Note that all parts of item 2 can be similarly verified for $\overline{\Sigma}(2p+\alpha-1)$. This proves Theorem \ref{4p+3}.
\end{proof}

Finally, to prove Theorem \ref{main2}, it remains to choose an $\alpha$ that satisfies the conditions of Theorem  \ref{4p+3}.
Let $n=4h+1$ and $\alpha=2h$; then gcd$(n,\alpha)=1$. Now since $n\geq 4p+3$, $h\geq p+1$ so
$2p+2\leq \alpha=2h\leq n-2h\leq n-2-2p$ so we are done.

\subsection{Globally simple $H(n;4p+3)$ when $n\equiv 0\md{4}$}

Finally it remains to prove Theorem \ref{main3}. Using Theorem \ref{Ladder2}, we can construct a suitable Heffter array $H(n;3)$ which merges with the support shifted Heffter array $H(n;4p,3)$ from Theorem \ref{main}, similarly to Theorem \ref{4p+3}.
In this process the diagonals of the Heffter array $H(n;3)$ become diagonals
$D_{2p+\alpha-\epsilon-1}$, $D_{2p+\alpha-1}$ and $D_{2p+\alpha+\epsilon-1}$ in the Heffter array $H(n;4p+3)$.
Then, so long as $2p+\alpha-\epsilon-1> 4p-2$ and  $2p+\alpha+\epsilon-1<n$, the partial sums will have all the same properties as in the $n\equiv 1\md{4}$ construction.
We thus have the following theorem.

\begin{theorem}
Let $n\equiv 0\md{4}$, $p\geq 0$ and $n\geq 4p+3$. Suppose $\epsilon\leq (n-4p)/2$ is coprime to $n$. If there exists an integer $\alpha$ coprime to $n$ such that $2p+\epsilon \leq \alpha\leq n-\epsilon-2p$, then there exists a globally simple Heffter array $H(n;4p+3)$.
\end{theorem}

For example, if $n\equiv 0\md{4}$ but $n\not\equiv0\md{12}$ and $n\geq 4p+8$, then choosing $\epsilon=3$ and
$\alpha=n/2-1$ yields a globally simple $H(n;4p+3)$.

The {\em Jacobsthal function} $j(n)$ is defined to be the smallest $m$ such that every sequence of $m$ consecutive integers contains an integer coprime to $n$. It was shown in \cite{Iw} that $j(n)=O(\log^2{n})$.
Thus for $n-4p$ sufficiently large we can choose $\epsilon=O(\log^2{n})$ and $\alpha=2p+O(\log^2{n})$ which are each coprime to $n$ and satisfy the inequalities of the above theorem. Thus Theorem \ref{main3} is true.

\section{Conclusion and Future Work}

As shown in \cite{ADDY} and \cite{DW},  an integer Heffter array $H(n;k)$ exists if and only if $nk\equiv 0,3 \md{4}$.
In this paper we have shown the existence of an integer Heffter array $H(n;k)$ which is globally simple whenever (a) $k\equiv 0\md{4}$; (b) $n\equiv 1\md{4}$ and $k\equiv 3\md{4}$; or (c) $n\equiv 0\md{4}$, $k\equiv 3\md{4}$ and $n\gg k$.
In future work we will show that in most cases (in particular when $n$ is prime), the array $H(n;4p+3)$ given in Section 4 has an ordering which is both simple and compatible.
As discussed in the introduction, this will yield biembeddings of cycle systems on orientable surfaces. We will also give lower bounds on the number of such non-isomorphic biembeddings.

\vspace{5mm}
\noindent{\bf Acknowledgment:} The fourth author would like to acknowledge support from
TUBITAK 2219 and the School of Mathematics and Physics, The University of Queensland, through the awarding of a Ethel Raybould Visiting Fellowship.


\begin{appendix}

{\Large\bf Appendix}

\begin{center}
Row partial sums for Example \ref{appie}. \\
\begin{scriptsize}
\begin{tabular}{|c||c|c|c|c|c|c|c|c|c|c|c|c|}
\hline
&\multicolumn{12}{c|}{Partial sum to diagonal}\\
\hline
Row Number&$D_0$&$D_1$&$D_2$&$D_3$&$D_4$&$D_5$&$D_6$&$D_7$&$D_8$&$D_9$&$D_{10}$&$D_{12}$\\
                           \hline
                           \hline
                           0 &85   &1    &150  &2    &215  &3    &-250 &-82  &-251 &-147 &-252    &0\\
                           \hline
                           1 &53   &1    &152  &2    &217  &3    &-222 &-56  &-223 &-121 &-224    &0\\
                           \hline
                           2 &55   &1    &154  &2    &219  &3    &-228 &-64  &-229 &-129 &-230    &0\\
                           \hline
                           3 &57   &1    &122  &2    &221  &3    &-234 &-72  &-235 &-137 &-236    &0\\
                           \hline
                           4 &59   &1    &124  &2    &223  &3    &-240 &-80  &-241 &-145 &-242    &0\\
                           \hline
                           5 &61   &1    &126  &2    &191  &3    &-246 &-88  &-247 &-153 &-248    &0\\
                           \hline
                           6 &63   &1    &128  &2    &193  &3    &-252 &-96  &-253 &-161 &-254    &0\\
                           \hline
                           7 &65   &1    &130  &2    &195  &3    &-224 &-70  &-225 &-135 &-226    &0\\
                           \hline
                           8 &67   &1    &132  &2    &197  &3    &-230 &-44  &-231 &-143 &-232    &0\\
                           \hline
                           9 &69   &1    &134  &2    &199  &3    &-236 &-52  &-237 &-151 &-238    &0\\
                           \hline
                          10 &71   &1    &136  &2    &201  &3    &-242 &-60  &-243 &-125 &-244    &0\\
                           \hline
                          11 &73   &1    &138  &2    &203  &3    &-248 &-68  &-249 &-133 &-250    &0\\
                           \hline
                          12 &75   &1    &140  &2    &205  &3    &-220 &-42  &-221 &-107 &-222    &0\\
                           \hline
                          13 &77   &1    &142  &2    &207  &3    &-226 &-50  &-227 &-115 &-228    &0\\
                           \hline
                          14 &79   &1    &144  &2    &209  &3    &-232 &-58  &-233 &-123 &-234    &0\\
                           \hline
                          15 &81   &1    &146  &2    &211  &3    &-238 &-66  &-239 &-131 &-240    &0\\
                           \hline
                          16 &83   &1    &148  &2    &213  &3    &-244 &-74  &-245 &-139 &-246    &0\\
                           \hline
                           \end{tabular}
                           \end{scriptsize}
                           \end{center}

  \begin{center}
Column partial sums for Example \ref{appie}. \\
\begin{scriptsize}
\begin{tabular}{|c||c|c|c|c|c|c|c|c|c|c|c|c|}
\hline
&\multicolumn{12}{c|}{Partial sum to diagonal}\\
\hline
Column Number&$D_0$&$D_1$&$D_2$&$D_3$&$D_4$&$D_5$&$D_6$&$D_7$&$D_8$&$D_9$&$D_{10}$&$D_{12}$\\
                           \hline
                           \hline
0 &85&33&186&66&287&99&-156&-2&-189&-103&-222&0\\
                           \hline
1 &53&-1&120&-2&187&-3&-230&-44&-229&-111&-228&0\\
                                                     \hline
2 &55&-1&122&-2&189&-3&-236&-52&-235&-119&-234&0\\
                                                     \hline
3 &57&-1&124&-2&191&-3&-242&-60&-241&-127&-240&0\\
                                                     \hline
4 &59&-1&126&-2&193&-3&-248&-68&-247&-135&-246&0\\
                                                     \hline
5 &61&-1&128&-2&195&-3&-254&-76&-253&-143&-252&0\\
                                                     \hline
6 &63&-1&130&-2&197&-3&-226&-50&-225&-117&-224&0\\
                                                     \hline
7 &65&-1&132&-2&199&-3&-232&-58&-231&-125&-230&0\\
                                                     \hline
7 &67&-1&134&-2&201&-3&-238&-66&-237&-133&-236&0\\
                                                     \hline
9 &69&-1&136&-2&203&-3&-244&-74&-243&-141&-242&0\\
                                                     \hline
10&71&-1&138&-2&205&-3&-250&-82&-249&-149&-248&0\\
                                                     \hline
11&73&-1&140&-2&207&-3&\em{-256}&-90&-255&-157&-254&0\\
                                                     \hline
12&75&-1&142&-2&209&-3&-228&-64&-227&-131&-226&0\\
                                                     \hline
13&77&-1&144&-2&211&-3&-234&-72&-233&-139&-232&0\\
                                                     \hline
14&79&-1&146&-2&213&-3&-240&-80&-239&-147&-238&0\\
                                                     \hline
15&81&-1&148&-2&215&-3&-246&-88&-245&-155&-244&0\\
                                                     \hline
16&83&-1&150&-2&217&-3&-252&-96&-251&-163&-250&0\\
                           \hline
                           \end{tabular}

                         \end{scriptsize}
                            \end{center}

\end{appendix}

 \end{document}